\documentclass[10pt]{article}
\newcommand{\keywords}{{\em Key words:\ }}
\newcommand{\subjclass}{{\em MSC: \ }}

\usepackage{xspace}
 
\usepackage{amsmath}
\usepackage{amsthm}
\usepackage{amssymb}

\usepackage{xspace}
 
\usepackage{amsmath}
\usepackage{amsthm}
\usepackage{amssymb}

\newtheorem{theorem}{Theorem}[section]
\newtheorem{proposition}[theorem]{Proposition}
\newtheorem{lemma}[theorem]{Lemma}
\newtheorem{corollary}[theorem]{Corollary}

\theoremstyle{definition}
\newtheorem{definition}[theorem]{Definition}
\newtheorem{example}[theorem]{Example}

\theoremstyle{remark}
\newtheorem{remark}{Remark}
\newtheorem*{assume}{Assumption}

\newcommand{\KK}{\mathbb{K}}        
\newcommand{\R}{\mathbb{R}}        
\newcommand{\C}{\mathbb{C}}      
\newcommand{\HH}{\mathbb{H}} 
\newcommand{\VV}{\mathfrak{V}}
\newcommand{\II}{\mathfrak{I}}
\newcommand{\JJ}{\mathcal{J}}

\newcommand{\MK}{\KK^{n\times n}} 
\newcommand{\MKD}{\KK^{2n\times 2n}} 

\newcommand{\HX}{\widehat{X}} 
\newcommand{\HD}{\widehat{D}^\prime} 

\newcommand{\I}{\mathbf{i}}
\newcommand{\J}{\mathbf{j}}
\newcommand{\K}{\mathbf{k}}

\newcommand{\g}{\mathfrak{g}}     

\newcommand{\diag}{\mathop{\mathrm{diag}}}    
\newcommand{\grad}{\mathop{\mathrm{grad}}}    
\newcommand{\Tr}{\mathop{\mathrm{Tr}}}          

\newcommand{\defeq}{\mathrel{\mathop:}=}

\newcommand\bigzero{\makebox(0,0){\rm{\huge0}}}

\newcommand{\vertmat}[1]{\multicolumn{1}{c|}{#1}}

\newcommand{\hps}{H.p.-s.\xspace}

\newcommand{\sech}{\mathop{\mathrm{sech}}}

\usepackage{url}

\title{Height functions on compact symmetric spaces}
\author{E. Mac\'{\i}as-Virg\'os \and M.J. Pereira-S\'aez}
\date{\today}

\begin{document}

\maketitle 

\begin{abstract} We consider height functions on symmetric spaces $M\cong G/K$ embedded in the associated matrix Lie group $G$. In particular we study the relationship between the critical sets of the height function   on $G$ and its restriction  to $M$. Also we prove that the gradient flow on $M$  can be integrated   by means of a generalized Cayley transform. This allows to obtain explicit local charts for the critical submanifolds. Finally, we discuss how to reduce the generic case to a height function whose ground hyperplane is orhogonal to a real diagonal matrix. This result  requires to prove the existence of a polar decomposition adapted to the automorphism defining $M$. Detailed examples are given.
 \end{abstract}



\maketitle

\tableofcontents

\section*{Introduction}
A common method in Morse theory is to embed a manifold in a suitable Euclidean space and examine the critical set of some height function. In particular, for symmetric spaces it is possible to construct isometric embeddings   in a unified manner \cite{KM2008,KOBAYASHI}.  Frankel \cite{FRANKEL} applied Morse theory to the classical Lie groups and Grassmannians ---described in terms of matrices---, and gave a complete description of the critical submanifolds, by using the trace functional as Morse function. Nicolaescu \cite{NICOLAESCU} also studied trace  functions on Grassmannians. Ramanujam \cite{RAMANUJAN_homogeneous,RAMANUJAM_symmetric,RAMANUJAM_classical}, Takeuchi \cite{TAKEUCHI_nice}, Duan \cite{DUAN}, and other authors  
investigated Morse functions on symmetric spaces by choosing a suitable length function. In \cite{VD}, Dynnikov and Veselov considered height functions on the natural embeddings of classical Lie groups and certain symmetric spaces. 

The generic formula for such a function is $h_X(A)=\Re\Tr(XA)$, where $\Re\Tr$ is the real part of the trace, and $X^*$ is perpendicular to the hyperplane taken as ground level.  A crucial point  in order to describe the critical set $\Sigma(h_X^G)$ in the Lie group $G$ is to reduce the matrix  $X$ to a diagonal matrix $D$ with non-negative real entries  by  taking some singular value decomposition $UDV^*$.  In fact, it can be proven that, on the compact Lie group $G$, the height function $h_X^G$ is a Morse function  (isolated non-degenerate critical points)  if and only if the singular values of $X$ are positive and pairwise different.  When $X^*X$ has repeated eigenvalues,  the function is of Morse-Bott type (that is, there are critical submanifolds instead of isolated points, with non-degenerate Hessian  in the transverse directions). An explicit description of the critical set of an arbitrary height function   on the Lie group $G$ appears in the authors' paper \cite{TMP}. 


For symmetric spaces, however, there is no systematic characterization of Morse functions. In fact,   the  results for Lie groups led in the past to study only real diagonal matrices.   As a consequence, the behaviour of the height functions on symmetric spaces might seem much more regular than  it actually is.
For instance, when  $M\cong G/K$ is embedded in the associated Lie group $G$, Ramanujam \cite{RAMANUJAM_symmetric}
 states that ``the critical submanifolds of $G/K$ are shown to be the intersection of the space $G/K$ and the critical submanifolds of $G$'' while Dynnikov and Veselov \cite{VD} write that  ``symmetric spaces [...]  are invariant by the gradient flow of the height function on the corresponding Lie groups'' and that ``the restricted flow coincides with the gradient flow of the [restricted] function''. As we shall see, this kind of result only holds in particular ---although important--- cases, but is no longer true for a generic height function on a symmetric space.
 
The situation is much more interesting indeed. We shall show that it is possible to compute the critical set  of the restricted function $h_X^M$ in terms of the matrix $\HX\defeq X^*+\sigma(X)$, where $\sigma$ is the involutive automorphism  of $G$ defining the symmetric space. Namely, we obtain the general formula
\begin{equation}\label{GF}
\Sigma(h_X^G)\cap M\subset \Sigma(h_X^M)=\Sigma(h^G_{\sigma(\HX)})\cap M.
\end{equation} 
Only under additional hypothesis the left inclusion will  be an equality.   Similarly, the gradient flow in $G$ preserves  the symmetric space when  $\sigma(X)=X^*$,
but not in the general case.  

Anyway,  we will be able to prove  a result of reduction of an arbitrary height function $h_X^M$ to the diagonal case. 
In order to obtain  that result we shall need to show the existence of  singular value decompositions  which are adapted to the symmetric space. This problem has been widely studied   \cite{KAAS,LAWSON}. But in contrast to the already known results, which are of local nature, our decomposition theorem is global 
and relies in the relationship between the critical points   and the polar form of the matrix $\HX$.  This  
result is of  interest by itself.

Gradient flows related to the
minimization of a distance function
have been widely studied too. Unlike most Riemannian manifolds, the gradient flows of the height functions on the classical Lie groups   can be integrated explicitly. The key idea, found independently by the authors \cite{TMP} and other authors  (\cite{VOLKOZ}, cited in \cite{VD}) is that this can be achieved by means of the so-called Cayley transform. In this paper we prove that, after a suitable generalization, it is also possible to do so in symmetric spaces. The formulae obtained contain all previous results as particular cases, and may  be of interest for applications in optimization and control theory \cite{HELMKEMOORE, HELMKESHAYMAN,WU}. 

Finally, we  prove  that  the generalized Cayley transform allows  to give explicit local charts for the critical submanifolds of  height functions defined on a symmetric space, a result which is completely new.  
Another consequence  is a method to cover the symmetric space by a finite family of contractible open sets, which is suitable for the computation of the  Lusternik-Schnirelmann category. In fact,  Morse-Bott functions can be effectively used  to determine the L.-S. category of some Lie groups and homogeneous spaces \cite{TMP, KM2011,MPSPN}.

 
The contents of the paper are as follows. 

In Section \ref{PRELIM} we fix some notations and basic definitions. We shall consider the symmetric space $M\cong G/K$ associated to the involutive automorphism $\sigma$ of the compact Lie group $G$, where $K$ is the fixed point subgroup. Then $M$ can be embedded in a natural manner in $G$ by means of the so-called Cartan embedding. 

In Section \ref{HEIGHT} we consider an arbitrary height function on $M$, which is defined by the formula $h_X^M(A)=\Re\Tr(XA)$, where the matrix $X^*$ is  orthogonal to the ground hyperplane. We then compute the gradient and the Hessian. In particular the point $A\in M$ is a critical point of $h_X^M$ if and only if $\HX = A \sigma(\HX) A$, where $\HX\defeq X^*+ \sigma(X)$. As a particular case we compute the critical points of the function $h_X^G$ on the Lie group $G$; this time the critical equation is $X^*=AXA$. Thus we obtain formula (\ref{GF}) and we show with an example that the first inclusion may be strict. 

In Section \ref{CAYLEY} we introduce a generalization of the Cayley transform between the Lie group $G$ and its Lie algebra $\g$. More precisely, each point $A\in G$ has a neighbourhood $\Omega_G(A)$ which is mapped diffeomorphically onto the tangent space $T_{A^*}G$. We then verify that this diffeomorphism passes naturally to the symmetric space $G/K\cong M\subset G$, thus providing a covering of $M$ by contractible open sets $\Omega_M(A)=\Omega_G(A)\cap M$. 

In Section \ref{LINEAR} we prove that the Cayley transform allows to linearize the gradient flow and to solve explicitly the gradient equation for symmetric spaces. This has important consequences for the knowledge of the critical set of $h_X^M$, since the Cayley transform will define a local chart  $\Omega_M(A)\cap \Sigma(h_X^M)$ which is modeled on the kernel of the Hessian. Finally we compare the flows associated to the function $h_X^G$ and its restriction $h_X^M$ to the symmetric space $M$.

In Section \ref{SVD} we prove --under some hypothesis, that the matrix $\HX$ admits a singular value decomposition  and a polar form which are adapted to the automorphism $\sigma$ defining the symmetric space. More explicitly, since $\sigma(\HX)=\HX^*$ we prove that there is a polar decomposition $\HX=S\Omega$ such that the orthogonal part verifies $\sigma(\Omega)=\Omega^*$ and the hermitian part verifies $\sigma(S)=\Omega^* S \Omega$. Analogously there is a SVD decomposition $\HX=UDV^*$ such that the matrix $\Theta=U^*\sigma(V)$ verifies $\sigma(\Theta)=\Theta^*$. Then the critical set of $h_X^M$ is diffeomorphic to the critical set of $h_D^{M^\prime}$, where $D$ is a real diagonal matrix and $M^\prime$ is a new symmetric space associated to the automorphism $\sigma^\prime(X)=\Theta \sigma(X)\Theta^*$. This simplifies considerably the study of a generic height function, as we show by a detailed final example.

\section{Preliminaries}\label{PRELIM}
We begin by establishing some notation and recalling several basic definitions and results.

\subsection{Lie groups of orthogonal type}
Let  $\KK$ be one of the algebras $\R$ (reals), $\C$ (complex) or $\HH$ (quaternions).
For the matrix $A\in\KK^{n \times n}$ we denote by $A^*$  its conjugate transpose. Let 
$$O(n,\KK)=\{A\in \MK\colon AA^*=I\}$$
be the compact Lie group of orthogonal (resp. unitary, symplectic) matrices and let $G$ be the connected component of the identity. Thus, $G$ will be either $SO(n)$, $U(n)$ or $Sp(n)$. In the real case we shall consider sometimes the whole Lie group $O(n)$.  The Lie algebra of $G$ is formed by the skew-symmetric (resp. skew-Hermitian) matrices,
$$\g=\{X\in\KK^{n\times n} \colon X+X^{*}=0\}.$$

The Riemannian metric induced on $G\subset \KK^{n\times n}$ by the usual inner product  $\langle X,Y\rangle=\Re \Tr (X^*Y)$ is bi-invariant ($\Re\Tr$ denotes the real part of the trace). 
 
\subsection{Compact symmetric spaces}
Let $\sigma \colon G\to G$ be an involutive automorphism and let $$K=\{B\in G \colon \sigma(B)=B\}$$ be the closed Lie subgroup of fixed points. 
 
\begin{assume}
In this paper we shall assume  that the automorphism  $\sigma$ is the restriction of an involutive automorphism  $\sigma: \KK^{n\times n}\to \KK^{n\times n}$ of unital algebras. We shall also assume that $\sigma(X^*)=\sigma(X)^*$ for all $X\in\MK$.
 These conditions are not  too restrictive; for instance, all the compact irreducible Riemannian symmetric spaces in Cartan's classification fullfill them (see \cite[p.~518]{HELGASON}).
\end{assume}

The Lie algebra of $K$ is $\mathfrak{k}=\{X\in\g\colon \sigma(X)=X\}$ and the tangent space $T_{[I]}G/K$ is isomorphic to  $\mathfrak{m} =\{X\in\g\colon \sigma(X)=-X\}$. Since $\mathrm{Ad}(k)(\mathfrak{m})=\mathfrak{m}$ for all $k\in K$,  there is an invariant Riemannian metric on $G/K$ which has null torsion and parallel curvature \cite{FOMENKO, HELGASON, ZILLER}.
The homogeneous space $G/K$ is  called a {\em globally symmetric compact space}. 

Let  $\gamma\colon G/K\hookrightarrow G$ be the embedding given by $\gamma([B])=B\sigma(B)^{-1}$. It is an isometry (up to the constant $2$).

\begin{proposition}[{\cite[p.~185]{FOMENKO}}]\label{COMPONENTE}
Assume that $G/K$ is connected. Then the image $M=\gamma(G/K)$ of $\gamma$ is  the connected component $N_I$ of the identity of the submanifold
\begin{equation}\label{ENE}
N=\{B\in G \colon \sigma(B)=B^{-1}\}.
\end{equation}
\end{proposition}

See \cite{KM2008} for several examples. The manifold $M$  will be called the {\em Cartan model} of the symmetric space $G/K$. The isometric action of $G$  induced by $\gamma$ on $M$ is given by $l_B^M(A)=BA\sigma(B)^{-1}$, for $B\in G$, $A\in M$.  

\begin{example} \label{GRUPO}
The Lie group $G$ itself is a symmetric space defined by the automorphism  $\sigma\colon G\times G\to G\times G$ with $\sigma(B_1,B_2)=(B_2,B_1)$. The fixed point set is the diagonal $\Delta$. The diffeomorphism $G\cong (G\times G)/\Delta$ is given by $B\cong [(B,I)]$. The submanifold $N\subset G\times G$ is formed by all pairs $(B,B^{-1})$.
\end{example}

\begin{proposition}\label{TANGENTE} For any point $A\in M$, the tangent space is
$$T_AM=\{Y\in\MK\colon YA^*+AY^*=0, \, \sigma(Y)=Y^* \}.$$
\end{proposition}
\begin{proof}

Since $A=\gamma([B])$ for some $B\in G$, it is
$T_AM=(l_B^M)_{*I}(T_I M)$,
so  $Y\in T_AM$ if and only if $(l_B^M)^{-1}(Y)=B^*Y\sigma(B)\in T_IM$
because $l_B^M$ is a linear action. The result follows from the definition of $\mathfrak{m}$.
\end{proof}

\section{Height functions on symmetric spaces}\label{HEIGHT}
First of all we shall consider an arbitrary height function on the vector space $\MK$. Then we shall study how it behaves when restricted to the Cartan model $M$ of a symmetric space. Finally  we shall obtain as a particular case the corresponding results for the Lie group $G$. 

The height function $h_X\colon \MK\to \R$ with respect to an hyperplane perpendicular to $X^*\in \MK$, where $X\neq 0$, is given, up to a constant, by $$h_X(Y)=\langle X^*,Y\rangle=\Re\Tr (XY).$$

\subsection{Gradient} Since $h_X$ is $\R$-linear, its gradient at $Y\in\MK$ is
$(\grad h_X)_Y=X^*$.
Let  $h^M_X\colon M\to \R$ be the restriction of  $h_X$ to the Cartan model $M\subset G\subset \MK$ of the symmetric space $G/K$. 

We denote  ${\HX}\defeq X^*+\sigma(X)$. Notice that $\sigma({\HX})={\HX}^*$.

\begin{lemma}\label{descomposicionenM}
The projection of $Z\in \MK$ onto $T_I G$ is its skew-symmetric part, $\frac{1}{2}(Z-Z^*)$. The projection of the latter onto $\mathfrak{m}=T_I M$ is its  skew-invariant part,  that is, $$\frac{1}{4}[(Z-\sigma(Z))-(Z^*-\sigma(Z)^*)].$$ 
\end{lemma}

\begin{proposition}\label{gradienteES} The gradient  of $h_X^M$ at any point $A\in M$ is  the projection of $\grad{h_X}$ onto $T_AM$, that is,
\begin{equation}
\label{GRADSYM}(\grad h_X^M)_A=\frac{1}{4}\left( {\HX}-A\sigma({\HX})A\right).
\end{equation}
\end{proposition}

\begin{proof}
Using Lemma \ref{descomposicionenM},
$$(\grad h_X^M)_I=\frac{1}{4}\left[(X^*-\sigma(X)^*)-(X-\sigma(X))\right].$$

Let $A=B\sigma(B)^{-1}=l_B^M(I)$. Since the translations $l_B^M$ are isometries we have $$(\grad h_X^M)_A=l_B^M\left(\grad (h_X^M\circ l_B^M )_I\right)=l_B^M\left( (\grad h^M_{\sigma(B)^*XB})_I\right).$$
Then we apply the action to $$\frac{1}{4}\left((B^{*}X^*\sigma(B)-\sigma(B)^*\sigma(X^*)B)-(\sigma(B^*)XB-B^*\sigma(X)\sigma(B))\right)$$
and we obtain
{\allowdisplaybreaks \begin{align*}
& \frac{1}{4}\left(X^*-B\sigma(B)^*\sigma(X)^*B\sigma(B)^*-B\sigma(B)^*XB\sigma(B)^*+\sigma (X)\right) \\
 =& \frac{1}{4}({\HX}-A{\HX}^*A).
\end{align*}}
\end{proof}

\begin{corollary} $A\in M$ is a critical point of $h_X^M$ if and only if ${\HX}=A\sigma({\HX})A$, where $\HX\defeq X^*+\sigma(X)$.
\end{corollary}

\begin{remark}\label{REMPERP} 
 Instead of height, one can consider the {\it distance to} $X^*$. Since
$$\vert A-X^*Ê\vert^2 = a h^M_X(A) + b,\quad a,b\in\R,$$ both functions have the same critical points in $M$. Geometrically,  these are the points where the line $\stackrel{\rightarrow}{AX^*}$ is perpendicular to $T_AM$ \cite[p. 35]{MILNOR}.
\end{remark}

\subsection{Hessian}
Let $A\in M$. Since $M$ and $G$ are submanifolds of $\MK$, in order to compute the Hessian we have to extend the vector field $\grad h_X^M$ to all matrices \cite[p.~72]{GALLOT}, taking for instance
$$\widetilde{(\grad h_X^M})_B=\frac{1}{4}\left( {\HX}-B\sigma({\HX})B\right).$$

The covariant derivative $\nabla$ in $\MK$ is the usual derivative.
Then for $W\in T_AM$ the Hessian $(\nabla_W^M \grad)_A$ is the projection of
{\allowdisplaybreaks  \begin{align*}
 (\nabla_W\widetilde{\grad})_A
 & =
\lim\limits_{t\to 0}\frac{1}{t}\left[(\widetilde{\grad})_{A+tW}-(\widetilde{\grad})_A \right]\\
& =\lim\limits_{t\to 0}\frac{1}{4t}\left[ {\HX}-(A+tW)\sigma( {\HX})(A+t W)-( {\HX}-A\sigma( {\HX})A)\right] \\
& = -\frac{1}{4}(A\sigma( {\HX}) W+ W\sigma( {\HX})A),
\end{align*}}
which belongs to $T_AM$. Then we have proved:
\begin{proposition}
The Hessian  $H (h_X^M)_A \colon T_AM \to T_AM$ of the height function $h_X^M \colon M \to \R$ is given by
$$H(h_X^M)_A(W)=-\frac{1}{4}\left(A\sigma({\HX})W+W\sigma({\HX})A\right).$$
\end{proposition}

An easy computation shows that: (i) $A$ is a critical point of $h_X^M$ if and only if the matrix $\HX^*A$ is Hermitian; (ii) the matrix $W$ belongs to the tangent space $T_AM$  if and only if  $WA^*$ is skew-Hermitian and $\sigma(W)=W^*$; (iii)  $W$ is in the kernel of the Hessian
if in addition the matrix $\HX^*W$ is Hermitian.

\begin{example}
Let us consider the complex Grassmannian  $U(2)/(U(1)\times U(1))$ defined by the automorphism $\sigma(A)=I_{1,1}A I_{1,1}$, where $I_{1,1}=  \begin{pmatrix}1&0\cr
0&-1\cr\end{pmatrix}$. The Cartan model is the sphere $S^2 \subset U(2)\cong S^3\times S^1$ formed by the matrices $\begin{pmatrix}s&-\overline{z}\cr
z&s\cr\end{pmatrix}$ where $(s,z)\in \mathbb{R}\times \mathbb{C}$ verifies $s^2+\vert z\vert^2=1$. 

Let us  take on $M$ the  function $h^M_X$ with
$X=\begin{pmatrix}
0& 0\cr
0 &1\cr
\end{pmatrix}$. Then 
$\HX=\begin{pmatrix}
0& 0\cr
0 &2\cr
\end{pmatrix}$ and the critical points are the two poles $\pm I$. The tangent space $T_{\varepsilon I}M$ is formed by the matrices
$W=\begin{pmatrix}0&z\cr -\overline{z}&0\end{pmatrix}$ with $z\in \C$. The Hessian is $(H h_X^M)_{\varepsilon I}(W)=(-\varepsilon/2)W$, so $h^M_X$ is a Morse function on $M$.  On the other hand, as we will see (cf. Section \ref{LIEGRUPO}, also \cite{TMP}), the critical set of $h^G_X$ on the Lie group $U(2)$  is formed  by the two circles $U(1)\times \{\pm 1\}$ of
matrices
$\begin{pmatrix}
\alpha& 0\cr
0 &\pm 1\cr
\end{pmatrix}$ with $\alpha\in\C$, $\vert \alpha \vert=1$. 
\end{example}

\subsection{Height functions on the Lie group}\label{LIEGRUPO}
In \cite{TMP} the authors studied the height functions on the Lie group $G$ of orthogonal type. Let us sketch how to recover those results by considering $G$ as the symmetric space of Example \ref{GRUPO}.

\begin{proposition}\label{GRADGRUPO} The gradient of  $h_X^G\colon G \to \R$ at $A\in G$ is
\begin{equation}\label{GRADG}
(\grad h_X^G)_A=\frac{1}{2}(X^*-AXA).
\end{equation}
\end{proposition}
\begin{proof}
First, the Cartan embedding $\gamma\colon G \to G\times G$ can be extended to a map 
$$\gamma\colon \MK\to \MK\times\MK=\MKD$$ by putting 
$$\gamma(X)=
\begin{pmatrix}
X&0\cr 0&X^*\cr
\end{pmatrix}.$$ 
This $\R$-linear map is an  isometry (up to the constant $1/2$). Let $ M=\gamma(G)$. According to Proposition \ref{gradienteES}, the gradient of $h^M_{\gamma(X)}$ at $\gamma(A)\in M$ is
$$\frac{1}{4}(\widehat{\gamma(X)}-\gamma(A)\sigma(\widehat{\gamma(X)})\gamma(A))=\frac{1}{2}\gamma(X^*-AXA).$$
On the other hand,
$$(\grad h_X^G)_A=2\gamma_{*A}^{-1}(\grad (h^G_X\circ \gamma^{-1}))_{\gamma(A)}=\gamma^{-1}\big(
(\grad h^M_{\gamma(X)})_{\gamma(A)}
\big)$$
and the result follows.
\end{proof}
\begin{remark}
Compare formulae  (\ref{GRADSYM}) and (\ref{GRADG}). A similar computation is valid {\em mutatis mutandi} for the Hessian \cite[p.~329]{TMP}, the gradient flow \cite[p.~330]{TMP} and the local structure of the critical set  \cite[p.~331]{TMP} in the group $G$. In all the formulae it is enough to substitute $2X^*$ for $\HX$. 
\end{remark}

From Propositions \ref{gradienteES} and  \ref{GRADGRUPO}  we obtain

\begin{corollary}\label{CRITSIM} Let $M\subset G$ be the Cartan model of the symmetric space $G/K$. Then the critical set in $M$ of the height function $h_X^M$ is 
$$\Sigma(h_X^M)=\Sigma(h^G_{\sigma({\HX})})\cap M.$$
\end{corollary}

\begin{corollary}$\Sigma(h^G_X)\cap M\subset\Sigma(h_X^M)$.
\end{corollary}
\begin{proof}
This is an obvious consequence of Remark \ref{REMPERP}. Alternatively, notice that $X^*=AXA$ implies $\HX=A\sigma(\HX)A$ when $\sigma(A)=A^*$.
\end{proof}

According to Ramanujam \cite[p.~219]{RAMANUJAM_symmetric}, when $X=I$ the critical points of the height function
$h_X^M$ are just the points of $\Sigma(h_X^G)$ that belong to $M$.  The same result is true when $X$ is a real diagonal matrix for the symmetric spaces studied by Duan \cite{DUAN} and Dynnikov-Veselov \cite{VD}. 

\begin{corollary}\label{SENCILLO} If  $\sigma(X)=X^*$,  then  the critical points of the height function $h_X^M$ on the symmetric space $M$ verify that
$\Sigma(h_X^M)=\Sigma(h_X^G)\cap M$.
\end{corollary}

However, in the generic case the preceding result no longer holds, as the following example shows.

\begin{example}\label{EjemploCritGM}
Let us consider the symmetric space $Sp(1)/U(1)$. It is defined by the automorphism  $\sigma(X)=-\I X\I$. The Cartan model $M$ is the sphere $S^2\subset Sp(1)=S^3$ formed by the unit quaternions $q=s+\J z$ such that  $s\in\R$, $z\in \C$, with $s^2+\vert z \vert^2=1$. Notice that $q$ has a null $\I$-coordinate.

Now we consider the height function $h_X$ with $X={\I+\J+\K}$.

(i) On the group $G=Sp(1)$, by Proposition \ref{GRADGRUPO} we have that the critical points of $h_X^G$ are the elements 
$q=t+x\I+y\J+z\K\in Sp(1)$  such that 
 $-(\I+\J+\K)\overline{q}=q(\I+\J+\K)$.
From this equation we deduce that 
$t = 0$ and
$x=y=z$.
But since $q$ is unitary, we have that the only critical points in $G=Sp(1)$ are
$$\Sigma(h_X^G)=\{\pm\frac{1}{\sqrt{3}}(\I+\J+\K)\}.$$
These two points are not in $M$ because they have a non-null $\I$-coordinate. 

(ii) Nevertheless, there are points of $M$ that are critical points for the height function restricted to the Cartan model $M\subset Sp(1)$ of $Sp(1)/U(1)$. This time, by Proposition \ref{gradienteES}, the condition for a point $q\in M$ to be critical for $h_X^M$ is  $\widehat{X}=q\sigma(\widehat{X})q$, where $\widehat{X}=X^*+\sigma(X)=-2(\J+\K)$. So, from the condition
$-2\overline{q}(\J+\K)=2(\J+\K)q$
we obtain that
$$\Sigma(h_X^M)=\{\pm\frac{1}{\sqrt{2}}(\J+\K)\}.$$
\end{example}

\section{Cayley transform}\label{CAYLEY}
In this Section we integrate explicitly the gradient flow and give local charts for the critical submanifolds.
\subsection{Generalized Cayley transform}
The following {\it generalized Cayley map} was defined by the authors in \cite{TMP}. 
Let $A\in G$, that is, $AA^*=I$. We consider the open set of matrices
$$\Omega(A)=\{X\in  \MK\colon A+X \text{ is invertible}\}.$$

\begin{definition} 
The {\em { Cayley tansform centered at $A$}}  is the map
$c_A\colon \Omega(A) \to \Omega(A^*)$ defined by
$$c_A(X)=(I-A^*X)(A+X)^{-1}=(A+X)^{-1}(I-XA^*).$$
\end{definition}

Its most interesting property   is that it is a diffeomorphism, with $c_A^{-1}=c_{A^*}$. 

\begin{remark}
When $A=I$ one obtains the classical Cayley map
$c(X)={(I-X)}{(I+X)^{-1}}$
which is defined for the matrices $X$ that do not have $-1$ as an eigenvalue. Notice that $c_A(X)=c(A^*X)A^*$.
\end{remark}

\begin{theorem}[{\cite[p.~328]{TMP}}]\label{CONTRACTIBLE} Let $\Omega_G(A)\defeq\Omega(A)\cap G$. The map $c_A$ induces a diffeomorphism 
$$\Omega_G(A)\cong T_{A^*}G.$$  
In particular, $\Omega_G(A)$ is a contractible open subspace of $G$.
\end{theorem}

\begin{proof}
We sketch the proof for the sake of completeness. We shall use  the properties of $c_A$ that appear in \cite[p.~327]{TMP}.

{ We first prove that $c_A(\Omega_G(A))\subset T_{A^*}G$}.
Let $B\in\Omega_G(A)$, then $B^*=B^{-1}$ and
$$c_A(B)^*=c_{A^*}(B^*)=c_{A^*}(B^{-1})=-Ac_A(B)A.$$
It follows that the matrix $c_A(B)A$ is skew-Hermitian, that is, $c_A(B)\in T_{A^*}G$.
 
{Now we prove that $T_{A^*}G\subset \Omega(A^*)$}. 
If there exists $Y\in T_{A^*}G$ such that $Y+A^*$ is not invertible then $\exists v\neq 0$ such that $(Y+A^*)v=0$, hence $A^*v=-Yv$ and $-v=AYv$.
That means that $-1$ is an eigenvalue of $AY$, which is impossible because  the eigenvalues of a skew-Hermitian matrix  must have a null  real part.
 
{Finally, we prove that $c_{A^*}(T_{A^*}G)\subset \Omega_G(A)$}.
Let  $Y\in T_{A^*}G\subset \Omega(A^*)$. Then $c_{A^*}(Y)\in G$. Indeed,
$AY\in T_IG$ is skew-Hermitian, hence $I-Y^*A^*$ is invertible, so
$c_{A^*}(Y)=c_A(Y^*)^*=[(A+Y^*)^{-1}(I-Y^*A^*)]^*$ is invertible. Moreover, as
$-A^*Y^*A^*=Y$, it happens that
$$Y^*=-AYA=-Ac_A(c_{A^*}(Y))A=c_{A^*}(c_{A^*}(Y)^{-1})$$
while
$$Y^* = c_A(c_{A^*}(Y))^*=c_{A^*}(c_{A^*}(Y)^*).$$
Since $c_{A^*}$ is injective it follows that
$
c_{A^*}(Y)^*=\left(c_{A^*}(Y)\right)^{-1}
$.
\end{proof}

\subsection{The Cayley map in symmetric spaces}
We shall verify in the next paragraphs that the preceding properties of the Cayley transform in the Lie group $G$ are naturally inherited by the Cartan model $M\subset G$ of the symmetric space $G/K$. 

\begin{lemma}\label{cAcompatible}
If $A\in G$ then
$c_{\sigma(A)}\circ\sigma=\sigma\circ c_A$
on $\Omega(A)$, or, equivalently,
$\sigma\circ c_{\sigma(A)}=c_A\circ\sigma$
on $\Omega(\sigma(A))$.
\end{lemma}
\begin{proof}
As the matrix $A$ is orthogonal, so is  $\sigma(A)$.  Also we know that, 
$\sigma(A^*)=\sigma(A)^*$.
Now, let  $X\in\Omega(A)$, that is, the matrix $A+X$ is invertible. Then $\sigma(A+X)=\sigma(A)+\sigma(X)$ is also invertible, hence $\sigma(X)\in\Omega(\sigma(A))$.

So we can compose $\sigma$ and $c_{\sigma(A)}$ and we have
\begin{align*} 
 c_{\sigma(A)}(\sigma(X))  & = \left(I-\sigma(A)^*\sigma(X)\right)\cdot\left(\sigma(A)+\sigma(X)\right)^{-1}\\
  & = \left(I-\sigma(A^*X)\right)\cdot\sigma(A+X)^{-1} \\
  & = \sigma(I-A^*X)\cdot\sigma(A+X)^{-1} \\ 
   & = \sigma(c_A(X)). \qedhere
\end{align*}
\end{proof}

\begin{theorem}\label{OmegaMcontractil}
Let $M\subset G$ be the Cartan model of the symmetric space $G/K$. Let $A\in M$. Then
$\Omega_M(A)\defeq\Omega(A)\cap M$
is a contractible open subspace of $M$.
\end{theorem}
\begin{proof}
From Theorem \ref{CONTRACTIBLE} we know that $\Omega_G(A)\cong T_{A^*}G$
can be contracted to $A$ by the contraction $\nu\colon \Omega_G(A)\times [0,1] \to \Omega_G(A)$  given by 
$$\nu(X,t)=c_{A^*}(tc_A(X)).$$ 
When $A\in M$ we shall use the same contraction for $\Omega_M(A)$. Then it is enough to prove that when $X\in M$ the contraction $\nu(X,t)$ remains in $M$ for all $t\in[0,1]$.  Since $M=N_I$ (Proposition \ref{COMPONENTE}), it suffices to prove that $\nu(X,t)\in N$ for all $t$, that is,
\begin{equation}\label{QUEDA}
\sigma(\nu(X,t))=\nu(X,t)^{-1}.
\end{equation}
This can be shown as follows.
First, since $A,X\in M\subset N\subset G$, we have  $\sigma (A)=A^{-1}= A^*$ and $\sigma(X)=X^{-1}$.
Moreover, let $\nu(t)=\nu(X,t)$; then, since $X\in\Omega_G(A)$, it is   $\nu(t)\in \Omega_G(A)$ (Theorem \ref{CONTRACTIBLE}), hence
 $\nu(t)^{-1}=\nu(t)^*\in\Omega_G(A^*).$
On the other hand, $\nu(t)\in \Omega_G(A)$ implies that $\sigma(\nu(t))\in \Omega_G(\sigma(A))=\Omega_G(A^*)$.
Finally, from Lemma \ref{cAcompatible} and the properties of the generalized Cayley transform \cite[p.~327]{TMP}, we obtain
{\allowdisplaybreaks \begin{align*}
c_{A^*}(\nu(t)^{-1}) & = -Ac_A(\nu(t))A \\
& = -Atc_A(X)A \\
& = tc_{A^*}(X^{-1})
\intertext{while}
 c_{A^*}(\sigma(\nu(t))& = \sigma(c_{\sigma(A^*)}(\nu(t))) \\
 & = \sigma(c_A(\nu(t))) \\
 & =\sigma(tc_A(X)) \\
 & = t\sigma(c_A(X)) \\
 & = tc_{\sigma(A)}(\sigma(X)) \\
 & = tc_{A^*}(\sigma(X)) \\
 & = tc_{A^*}(X^{-1}),
\end{align*}}
and the result follows because $c_{A^*}$ is injective.
\end{proof}

\begin{remark} The existence of a covering by contractible open sets is of interest when computing the  {\it Lusternik-Schnirelmann category} of $M$ \cite{TMP, KM2011} .
\end{remark}

\section{Linearization of the gradient}\label{LINEAR}
The generalized Cayley transform allows to linearize the gradient flow of any height function on a symmetric space. The results in this Section generalize Volchenko and Kozachko's result (\cite{VOLKOZ}, cited in \cite{VD}) for the classical Cayley transform in a Lie group.

\subsection{Gradient equation}

\begin{theorem}\label{theoremBeta}
Let $h_X^M$ be an arbitrary height function on the symmetric space $M$. Let $A$ be a critical point. Then the solution of the gradient equation 
\begin{equation}\label{ECUACION}
4\alpha^\prime={\HX}-\alpha\sigma({\HX})\alpha,
\end{equation}
with initial condition $\alpha_0\in\Omega_M(A),$ is the image by the Cayley transform $c_{A^*}$ of the curve
\begin{equation}\label{BETA}
\beta(t)=\exp(\frac{-t}{4}A^*{\HX})\beta_0 \exp(\frac{-t}{4}{\HX}A^*),
\end{equation}
where ${\HX}=X^*+\sigma(X)$ and $\beta_0=c_A(\alpha_0)\in T_{A^*}M.$
\end{theorem}

\begin{proof}
First of all, let us prove that $\beta(t)\in T_{A^*}M$ (cf. Proposition \ref{TANGENTE}) for all $t\in\R$. 
Notice that   being $A$   a critical point implies that $A{\HX}^*={\HX}A^*$, so  this matrix  is Hermitian. 
Now, using   that $A\exp(XA)=\exp(AX)A$ for arbitrary matrices $A, X$, and that $\beta_0$ is in  $T_{A^*}M$, we have
{\allowdisplaybreaks \begin{align*}
A^*\beta(t)^* 
=&A^*\exp(\frac{-t}{4}A{\HX}^*)\beta_0^* \exp(\frac{-t}{4} {\HX}^*A)\\
 =&A^*\exp(\frac{-t}{4}{\HX}A^*)\beta_0^* \exp(\frac{-t}{4} {\HX}^*A) \\
 =&\exp(\frac{-t}{4}A^*{\HX})A^*\beta_0^* \exp(\frac{-t}{4} {\HX}^*A)\\
 =&\exp(\frac{-t}{4}A^*{\HX})(-\beta_0 A) \exp(\frac{-t}{4} {\HX}^*A) \\
 =&-\exp(\frac{-t}{4}A^*{\HX})\beta_0 \exp(\frac{-t}{4} A{\HX})A\\
 =&-\exp(\frac{-t}{4}A^*{\HX})\beta_0 \exp(\frac{-t}{4} {\HX}A^*)A\\
 =&-\beta(t)A. 
\end{align*}}
Besides, we know that $\sigma$ preserves the product, so  
\begin{equation}\label{sigmabeta}
\sigma(\beta(t))=\sigma  \big(\exp(\frac{-t}{4}A^*{\HX})\big)\sigma(\beta_0)
\sigma\big( \exp(\frac{-t}{4}{\HX}A^*)\big). 
\end{equation}
Remember that $\sigma({\HX})=({\HX})^*$ and that $\sigma(A)=A^{-1}=A^*$ because $A\in M\subset N$. Making use again of the algebraic properties of $\sigma$  and knowing that $\beta_0\in T_{A^*}M$ implies $\sigma(\beta_0)=\beta_0^*$, we have that (\ref{sigmabeta}) equals
{\allowdisplaybreaks
\begin{align*}
&\exp\big(\frac{-t}{4}\sigma(A^*{\HX})\big)\beta_0^*\exp\big(\frac{-t}{4}\sigma({\HX}A^*)\big)\\
=&\exp\big(\frac{-t}{4}A{\HX}^*\big) \beta_0^* \exp\big(\frac{-t}{4}{\HX}^*A\big)\\
=&\beta(t)^*.
\end{align*}}
This proves that $\beta(t) \in T_{A^*}M.$

Now we shall verify that $c_{A^*}(\beta)$ is a solution of  Equation (\ref{ECUACION}). By differentiating $\beta$ we have that
\begin{equation}\label{BETAPRIMA}
\beta^\prime =
(-1/4)\left( \beta{\HX}A^*+A^*{\HX}\beta\right).
\end{equation}
Let
$\alpha=c_{A^*}\circ\beta=(I-A\beta)(A^*+\beta)^{-1}$,
then
$\alpha\cdot(A^*+\beta)=I-A\beta$.
Taking derivatives,
$\alpha^\prime \cdot (A^*+\beta)+\alpha\beta^\prime=-A\beta^\prime$,
that is,
$\alpha^\prime \cdot (A^*+\beta)= -(A+\alpha)\cdot\beta^\prime$.

It is not hard to check that the inverse of  $A^*+\beta$ is 
$$\frac{1}{2}(A+c_{A^*}(\beta))=\frac{1}{2}(A+\alpha).$$ 
Then 
\begin{equation}\label{UFF}
\alpha^\prime= -\frac{1}{2}(A+\alpha)\beta^\prime(A+\alpha)
=\frac{1}{8}(A+\alpha)(\beta{\HX}A^*+A^*{\HX}\beta)(A+\alpha).
\end{equation}
Moreover,
$$\beta=c_A(\alpha)=(I-A^*\alpha)(A+\alpha)^{-1}=(A+\alpha)^{-1}(I-\alpha A^*)$$
implies that
$\beta(A+\alpha)=I-A^*\alpha$
and
$(A+\alpha)\beta=I-\alpha A^*$.
Therefore, using Equation (\ref{UFF}), we have that
{\allowdisplaybreaks \begin{align*}
8\alpha^\prime&=(A+\alpha)\beta{\HX}A^*(A+\alpha)+(A+\alpha)A^*{\HX}\beta(A+\alpha)\\
&=(I-\alpha A^*){\HX}A^*(A+\alpha)+(A+\alpha)A^*{\HX}(I-A^*\alpha) \\
&=2({\HX} -\alpha A^*{\HX}A^* \alpha) \\
&=2({\HX}-\alpha \sigma({\HX}) \alpha).   
\end{align*}}
\end{proof}

By using the definitions $\beta_0=(I-A^*\alpha_0)(A+\alpha_0)^{-1}$ and $c_{A^*}(\beta(t))=A(A^*-\beta(t))(A^*+\beta(t))^{-1}$, and the property $A^*\exp(\HX A^*)A=\exp(A^*\HX)$, one finally obtains an explicit formula for the solution, namely
\begin{eqnarray}\label{FLUJOHIPERB}
\alpha(t)&=&A\big(\sinh(\frac{t}{4}A^*\HX)+\cosh(\frac{t}{4}A^*\HX)A^*\alpha_0\big)\\
&&\quad \times \big(\cosh(\frac{t}{4}A^*\HX)+\sinh(\frac{t}{4}A^*\HX)A^*\alpha_0\big)^{-1}. \nonumber
\end{eqnarray}

This formula, for the particular case of a Lie group $G$, the classical Cayley transform $c_I$ and the particular height function $h^G_D$ where $D$ is a real diagonal matrix  is due to Dynnikov and Vesselov \cite{VD}.

\subsection{Local structure of the critical set}\label{CartaLocalCriticos}
Giving a local chart for the critical  set of a  Morse-Bott function  is another new application of the generalized Cayley transform.

Let $h_X^M(A)=\Re\Tr(XA)$ be a height function on the symmetric space $M$. Let us denote by $\Sigma(h_X^M)$ the critical set. Given a critical point $A\in\Sigma(h_X^M)$, let  $S^M\!(A)$ be the vector space
\begin{equation}\label{MODELO}
S^M\!(A)=\{\beta_0\in T_{A^*}M\colon A^*{\HX}\beta_0+\beta_0{\HX}A^*=0\}.
\end{equation}
Observe that $S^M\!(A)$ is  isomorphic to the kernel of the Hessian  $H (h_X^M)_A$. 

\begin{theorem}\label{DifeoEESS}
The generalized Cayley transform induces a diffeomorphism 
$$c_{A^*}\colon S^M\!(A)\to \Sigma(h_X^M)\cap \Omega_M(A).$$
\end{theorem}
\begin{proof}

From Theorem  \ref{theoremBeta},  we know that  $c_{A^*}(\beta_0)$  is a critical point if and only if $\beta^\prime(t)=0$ for all $t$, that is, the curve $\beta(t)$ is constant. Let us see that this is equivalent to the condition $\beta_0\in S^M(A)$. 

If $\beta(t)=\beta_0$ for all $t$ then   $\beta^\prime(0)=0$, which implies from Equation (\ref{BETAPRIMA}) that $A^*{\HX}\beta_0+\beta_0{\HX}A^*=0.$

Conversely, from $A^*{\HX}\beta_0=-\beta_0{\HX}A^*$ it follows that
$\exp(-\frac{t}{4}A^*{\HX})\beta_0=\beta_0 \exp(\frac{t}{4}{\HX}A^*)$, that is, 
$\beta(t)=\beta_0$ (Equation (\ref{BETA})).
\end{proof}

\subsection{Relationship between gradient flows}
As in Corollary \ref{SENCILLO}, if one restricts to the case $X=I$ or more generally to   $X=D$ a  real diagonal matrix such that $\sigma(D)=D$, then
the gradient flow of $h_X^G$ will be tangent to the symmetric space $M$, embedded into $G$. That means that the restricted flow coincides with the gradient flow of the height function restricted to $M$, with respect to the induced metric, as in \cite[Cor. 2.1]{VD}. But in general
 the symmetric space will not be invariant under  the gradient flow. In fact,  the gradient flow of $h_X^G$ could  even be tranverse to $M$.

\begin{example} Let us consider again Example \ref{EjemploCritGM}.
Take the critical point $A=(1/\sqrt{2})(\J+\K)\in M$. According to formula (\ref{FLUJOHIPERB}) the gradient flow line  of $h_X^M$ passing through $\alpha_0=1$ is given by
$$\alpha^M(t)=\sech{t\sqrt{2}} - \J (\tanh t\sqrt{2} )\frac{1-\I}{\sqrt{2}}\in M.$$
On the other hand, the flow line of $h_X^G$ in the group $G$ passing through the same point $\alpha_0=1$ is 
$$\alpha^G(t)=  
\sech(t\sqrt{3})-\tanh(t\sqrt{3})\frac{\I+\J+\K}{\sqrt{3}}
$$
where we have chosen the Cayley transform corresponding to the critical point $A=(1/\sqrt{3})(\I+\J+\K)$. Notice that $\alpha^G(t)\notin M$ for $t\neq 0$.
\end{example}

\section{Polar decompositions and critical set}\label{SVD}

The study of Morse-Bott functions can be considerably simplified by means of the so-called {\em singular value decomposition}. Also there is an interesting relationship between polar forms and critical points. The main result in this Section will be the existence of decompositions which are adapted to the automorphism $\sigma$ that defines the symmetric space.
\subsection{Singular value decomposition and polar form}

The following constructions are well known \cite{HORN}. For the quaternionic case see \cite{LORING,ZHANG_svd}.

Let ${Y}\in\MK$ be an arbitrary matrix. Since the matrix $YY^*$  is  Hermitian positive-semidefinite  (hereafter abbreviated to \hps),  its eigenvalues are real and non-negative, say $0,t_1^2,\dots,t_k^2$,
 with multiplicities $n_0,n_1,\dots,n_k$, respectively.  Let us consider the  matrix with diagonal blocks
\begin{equation}\label{DIAG}
D=\left(\begin{array}{cccc}
\vertmat{0_{n_0}}&&&\\ \cline{1-1}
&\vertmat{t_1I_{n_1}}&&\\ \cline{2-2}
&&\ddots\\
&&&\vertmat{t_kI_{n_k}}\\ \cline{4-4}
\end{array}\,\,\right)
 \in\MK,
\end{equation}
where 
$0<t_1<\dots<t_k$ and $ n_0+\cdots+n_k=n$. 
Then there exist orthogonal  (resp. unitary, symplectic) matrices $U$, $V$ such that
 $Y=UDV^*$ ({\em singular value decomposition}).  As a consequence we have the   {\em  left polar decomposition} $Y=S\Omega$, where the matrix $\Omega=UV^*$ is orthogonal  and the matrix $S=UDU^*$ is \hps In fact, $S$   is uniquely determined as the only \hps square root of $YY^*=UD^2U^*$.

Analogously we have the {\em right polar decomposition} $Y=\Omega S^\prime$, where $S^\prime=VDV^*=\Omega^* S \Omega$ is the only \hps square root of $Y^*Y=VD^2V^*$. We shall denote $S=(YY^*)^{1/2}$ and $S^\prime=(Y^*Y)^{1/2}$ the \hps matrices above. 
Notice that the orthogonal part $\Omega$ in the polar decompositions is not unique, excepting when $Y$ is invertible. 

Conversely, from any polar decomposition $Y=S\Omega$  it is possible to deduce a singular value decomposition because $S=(YY^*)^{1/2}$  is diagonalizable, say $S=UDU^*$. Then taking $V=\Omega^* U$ gives  $Y=UDV^*$.

\begin{lemma}\label{E2D2}Let $\Sigma$ be an Hermitian square root of $YY^*$. Then there exists some orthogonal matrix $W$ such that $S=WDW^*$ and $\Sigma=W\Delta W^*$, where $\Delta$ is a real diagonal matrix verifying $\Delta^2=D^2$ (that is, $\Delta$ only differs from $D$ by the signs of the entries).
\end{lemma}

\begin{proof} Since $\Sigma$ commutes with $S^2=\Sigma^2$, they can be simultaneously diagonalized. Moreover, the ordering of the diagonal of $D$ can be recovered by conjugation with a real orthogonal matrix.
\end{proof}

\subsection{Critical points}
 Let us consider the function $h_{Y}^G(A)=\Re\Tr(YA)$. According to Proposition \ref{GRADGRUPO}, each critical point $A\in G$ determines a decomposition  $Y=(YA)A^*$ of $Y$, where $\Sigma=YA$ is Hermitian but no necessarily positive-semidefinite. Let us call {\em almost polar} such a decomposition. Since $\Sigma^2=YY^*$ we know from Lemma \ref{E2D2} that $\Sigma=W\Delta W^*$, where
$$\Delta=D\cdot
\left(\begin{array}{cccc}
\vertmat{0_{n_0}}&&&\\ \cline{1-1}
&\vertmat{E_1}&&\\ \cline{2-2}
&&\ddots&\\
&&&\vertmat{E_k}\\ \cline{4-4}
\end{array}\,\,\right)
 \in\R^{n \times n}
$$
and
$$E_i=\begin{pmatrix}
\varepsilon_i^1&&\cr 
&\ddots&\cr
&&\varepsilon_i^{n_i}\cr
\end{pmatrix}, \quad \varepsilon_i^j=\pm 1.$$

Then the value of the function at the critical point $A$ is
\begin{equation}\label{VALOR}
h_{Y}^G(A)=\Re\Tr(\Sigma)=t_1\Tr E_1+\cdots +t_k\Tr E_k.
\end{equation}

\begin{remark}When $\KK=\R$, the orthogonal part of the polar decomposition may not be in the connected component $G=SO(n)$ of the identity.  In order to avoid cumbersome technicalities in the next Proposition, when dealing with the real case  we should consider that the group is $O(n)$. \end{remark}

\begin{proposition}\label{MINIMO}The critical point $A\in G$  is a global maximum of $h_{Y}^G$ if and only if the decomposition $(YA)A^*$ is a true polar decomposition (i.e. the Hermitian matrix $\Sigma=YA$ is positive-semidefinite).
\end{proposition}

\begin{proof} 
If $A$ is a global maximum for $h_{Y}^G$, then it is a critical point where the function value is that given in (\ref{VALOR}). But since $Y$ actually has a polar decomposition $\Omega S$ it happens that the maximum value is 
\begin{equation}\label{MAX}
\Re\Tr(S)=\Re\Tr(D)=n_1t_1+\cdots+n_kt_k.
\end{equation}
Comparing both values it follows that all $\varepsilon_{ij}=+1$, that is, $\Delta=D$ and $\Sigma=S$. Notice that the global maximum may not be unique due to the non-uniqueness of $\Omega$ when $Y$ is singular.

Conversely, if $YA$ is positive semi-definite then the critical value $h_X(A)$ is that given in (\ref{MAX}). This value is then bigger than all the other critical values, by formula (\ref{VALOR}). But since $G$ is compact, the function actually has a global maximum, which is a critical point. Then it must be the point $A$.
\end{proof}

\begin{remark} When $Y=SA^*$ is a polar decomposition it is well  known that $A$ maximizes the distance of $Y$ to the orthogonal matrices. In the same way, it maximizes the function  $\Re\Tr(YB)$, $B\in G$.
\end{remark}

\subsection{Polar decompositions in a symmetric space} According to Corollary \ref{CRITSIM}, the critical points of $h_X^M$ are the critical points of $h_{\sigma(\HX)}^G$ that lie in $M$. This is why we are interested in the matrix $\HX=X^*+\sigma(X)$. It verifies $\sigma(\HX)=\HX^*$. 

\begin{assume}
Let $\HX=UDV^*$ be a singular value decomposition.
{We shall assume that $\sigma(D)$ is positive-semidefinite.}
This hypothesis is not too restrictive. For the symmetric spaces in Cartan's classification it happens that either $\sigma(D)=D$ or $\sigma(D)=-JDJ$, with 
$J=\begin{pmatrix}
0&-I\cr
I&0
\end{pmatrix}$, which has the effect of permuting the entries of $D$ in two blocks. The latter case corresponds to the symmetric spaces  $U(2n)/Sp(n)$ and $SO(2n)/U(n)$.
\end{assume}

\begin{theorem}[adapted polar decomposition]\label{POLARSIM}Let $Y\in\MK$ be a matrix such that $\sigma(Y)=Y^*$. Assume that $\sigma(D)$ is positive semi-definite for the matrix $D$ of singular values of $Y$. Then there exists a polar decomposition $Y=S\Omega$ such that $\sigma(\Omega)=\Omega^*$ and $\sigma(S)=\Omega^* S \Omega$.
\end{theorem} 

\begin{proof}
{First, let us assume that $Y$ is invertible}. Then the polar decompositions $Y=S\Omega$ (left) and $Y=\Omega(\Omega^* S \Omega)$ (right) are unique. Clearly the matrix $\sigma(\Omega)$ is orthogonal, while $\sigma(S)$ is Hermitian. Moreover $\sigma(S)=\sigma(U)\sigma(D)\sigma(U)^*$, hence it is \hps by the assumption.  Then $Y=\sigma(Y)^*=\sigma(\Omega)^*\sigma(S)$ is a (right) polar decomposition, so by uniqueness $\sigma(\Omega)=\Omega^*$ and $\sigma(S)=\Omega^* S \Omega$.

We now deal with the non-invertible case.
Let $\VV$ be the real vector space 
$$\VV=\{Y\in \MK \colon \sigma(Y)=Y^*\}.$$
Then $\VV\cap G$ is the manifold $N$ defined in Proposition \ref{COMPONENTE}. Let $\delta\colon \VV \times G \to \R$ be the continuous function
$$\delta(Y,A)=\Re\Tr(Y A).$$
According to the first part of the proof and to Proposition \ref{MINIMO}, we know that for each invertible matrix $Y$ the function $\delta(Y,-)\colon G \to \R$ has a unique global maximum $\Omega$ which belongs to $N$. Then from Lemma \ref{DENSO} below it follows  that all functions $\delta(Y,-)$, $Y\in\VV$, have at least one maximum in $N$. But maxima correspond to polar decompositions by Proposition \ref{MINIMO}, so the result follows.

\end{proof}

\begin{lemma}\label{DENSO}Let $\II\subset \VV$ be the dense subpace of invertible matrices. Assume that each map $\delta(Y,-)$, with $Y\in\II$,  has a unique global maximum which belongs to $N$. Then all maps $\delta(Y,-)$, with $Y\in\VV$, have at least one maximum which belongs to $N$.\end{lemma} 
\begin{proof}
The invertible matrices form a dense subset in $\VV$, because for any $Y\in\VV$ it is
$$\sigma(Y+\varepsilon I)=\sigma(Y)+\varepsilon \sigma(I)=(Y+\varepsilon I)^*.$$

Let $Y_0$ be a non invertible matrix in $\VV$. Let $A_0\in G$ be any global maximum of the function $\delta(Y_0,-)\colon G \to \R$ and let $t_0=\delta(Y_0,A_0)$. Then there exists a sequence $(Y_n,A_n)$ converging to $(Y_0,A_0)$, where  
$Y_n\in\VV$ is invertible and $A_n\in G$, such that $t_n=\delta(Y_n,A_n)$ converges to $t_0$. Let $N_n\in G$ be the sole global maximum of $\delta(Y_n,-)$. Since $N$ is compact, there exists a subsequence $(N_n)$ which converges to some $N_0\in N$. Then
taking limits in $$\delta(Y_n,A_n)\leq  \delta(Y_n,N_n)$$
we obtain that
$$t_0= \delta(Y_0,A_0)\leq \delta(Y_0,N_0)$$
hence $N_0\in N$corresponds to the maximum $t_0$.
\end{proof}


\begin{corollary} Under the assumption, there exists a right polar decomposition $Y= \Omega S^\prime$ such that $\sigma(\Omega)=\Omega^*$ and $\sigma(S^\prime)=\Omega S^\prime \Omega^*$.
\end{corollary}

\begin{corollary}[adapted SVD]\label{ADAPT}  There exists a singular value decomposition $Y=UDV^*$ such that
the matrix $\Theta=U^*\sigma(V)$ verifies $\sigma(\Theta)=\Theta^*$.
\end{corollary}

\begin{proof}Let us take the polar decomposition $Y=S\Omega$ given by Theorem \ref{POLARSIM} and the corresponding singular value decomposition $Y=UDV^*$.
Then  
$\sigma(UV^*)=\sigma(\Omega)=\Omega^*=VU^*$.
\end{proof}

\section{Height functions associated to real diagonal matrices}\label{DIAGONAL}
In this section we show that, under mild hypothesis,  the study of any height function $h^M_X$ on the symmetric space $M$ can be reduced to the particular case $h^{M^{\prime}}_D$ where $D$ is a real non-negative diagonal matrix and $M^\prime\subset G$ is a symmetric space diffeomorphic to $M$. 

\subsection{Reduction to the diagonal case}
We begin with the following Lemma, which is a straightforward exercise.

\begin{lemma}\label{DESPLAZ} Let $\Theta$ be an orthogonal matrix such that $\sigma(\Theta)=\Theta^*$. Let $\sigma^\prime\colon G \to G$ be the involutive automorphism $\sigma^\prime(X)=\Theta \sigma(X)\Theta^*$ and let $K^\prime$ be its fixed point subgroup. Then the manifolds $N,N^\prime$ associated to $\sigma,\sigma^\prime$ respectively by formula (\ref{ENE}) verify that $N^\prime=\Theta N$.
\end{lemma}

Let $h_X^M \colon M \to \R$ be a height function on the symmetric space $M=G/K$ and let 
$\HX=X^*+\sigma(X)$.  Let $\HX=UDV^*$ be an adapted singular value decomposition as in the preceding Section (recall that we are assuming that $\sigma(D)$ is positive semi-definite).
Then, from 
$$\sigma(V)\sigma(D)\sigma(V)^*=\sigma(S)=\Omega S \Omega^*=UDU^*$$ it follows that $\sigma(D)=\Theta^* D\Theta$. Now, from Lemma \ref{DESPLAZ} we have that $\sigma^\prime(D)=D$, which implies that the matrix $\HD\defeq D^*+\sigma^\prime(D)$ equals $2D$, while $\sigma^\prime(\HD)=2D$.

Let $M^\prime$ be the symmetric space associated to $\sigma^\prime$. Hence (Proposition \ref{gradienteES}) the gradient of the height function $h_D^{M^\prime}\colon M^\prime \to \R$ at the point $B\in M^\prime$ is
\begin{equation}\label{GRADIAG}
(\grad h_D^{M^\prime})_B=\frac{1}{2}(D-BDB).
\end{equation}
 
 \begin{proposition}\label{RelCrit}
Assume $M=N$. Then the point $A\in M$ is a critical point of $h_X^M$ if and only if $U^*AV\in M^\prime$ is a critical point of $h_D^{M^\prime}$, where  ${\HX}=UDV^*$ is an adapted SVD. 
\end{proposition}
\begin{proof}
In fact,  by Proposition \ref{gradienteES} and formula (\ref{GRADIAG}) we have that
$$(\grad h_X^M)_A=2U((\grad h_D^{M^\prime})_{U^*AV})V^*$$
because $V^*\sigma(U)=\Theta=U^*\sigma(V)$.
Now, $U^*AV\in N^\prime$ because
$$\sigma^\prime(U^*AV)=V^*\sigma(U)\sigma(U^*AV)\sigma(V)^*U=V^*A^*U=(U^*AV)^*.$$ 
Finally, $M=N$ implies   $\Theta\in M$ and $M^\prime=N^\prime$. \end{proof}

Analogously we have for the Hessian that
$$(Hh_X^M)_A(W)=2U\left( (Hh_D^{M^\prime})_{U^*AV}(U^*WV)\right)V^*.$$

\subsection{Description of the critical set}

According to the previous section we can assume that $X=D$ is a diagonal matrix like in (\ref{DIAG}), and that $\sigma(D)=D$, hence $\widehat{D}=2D$.
\begin{lemma}
Let $A$ be a critical point of $h_D^G$, that is, $D=ADA$.  Then $A$ can be decomposed into boxes, of size $n_0,n_1,\dots,n_k$ respectively,
$$A=\left(\begin{array}{cccc}
\vertmat{A_0}&&&\bigzero\\ \cline{1-1}
&\vertmat{A_1}&&\\ \cline{2-2}
\bigzero&&\ddots&\\
&&&\vertmat{A_k}\\ \cline{4-4}
\end{array}\,\,\right)$$
 such that $A_0A_0^*=I$ and $A_i^2=I$, $A_i=A_i^*$, for $1\leq i \leq k$. 
\end{lemma}
\begin{proof} 
Write boxes
$$A=\left(\begin{array}{cccc}
{A_0}&A_{01}&A_{02}&\dots\\ 
A_{10}&{A_1}&A_{12}&\dots\\ 
&&\ddots&\\
&&&{A_k}\\
\end{array}\,\,\right)
$$

From the condition $DA^*=AD$, it follows that 
$A_{0j}=0$ and $A_{i}=A_{i}^*$, $i, j=1,\dots,k$. Moreover $t_iA_{ji}^*=t_jA_{ij}$ for $i<j$. From the conditions $AA^*=I=A^*A$ it follows that
$A_0A_0^*=I$ and $A_{i0}=0$, $i=1,\dots,k$. Also
$$A_{11}^2+A_{12}A_{12}^*+\cdots+A_{1k}A_{1k}^* =I = A_{11}^2+A_{21}^*A_{21}+\cdots+A_{k1}^*A_{k1}$$
that implies
$$(t_2^2-t_1^2)A_{12}A_{12}^2+\cdots+(t_k^2-t_1^2)A_{1k}A_{1k}^2=0$$
from which it follows easily that
$A_{12}=\dots=A_{1k}=0$. For the other rows the computation is similar .
\end{proof}
\begin{remark} Recall that
$\Sigma(h_D^M)=\Sigma(h_D^G)\cap M$. A precise description of the critical set of $h_D^G$ can be deduced from the previous Lemma, see \cite[p.~329]{TMP} and \cite{MPSPN}. Then $h_D^G$ is a Morse function if and only if $\dim S^M(A)=0$ (cf. Section \ref{CartaLocalCriticos}), which is equivalent to  $n_0=0$ and $n_1=\dots=n_k=1$.
\end{remark}

From Proposition \ref{GRADGRUPO} we obtain the following Corollary, ---this result appears in Brockett's paper \cite[p.~771]{BROCKETT}, where it is attributed to Shayman, see  also \cite[p.~99]{HELMKEMOORE}.

\begin{corollary} The height function $h^G_X$ in the Lie group $G$ is a Morse function if and only if the singular values of the matrix $X$ are positive and pairwise different.
\end{corollary}
 
  The following Example illustrates all the results in this Section.
\subsection{Final example}

 Let $G=Sp(2)$ and $\sigma(A)=-\I A \I$, then the fixed point subgroup is $K=U(2)$. The symmetric space $G/K=Sp(2)/U(2)$ can be identified with the manifold of complex structures on   $\HH^2$ which are compatible with the hermitian product. These are the matrices $\JJ\in Sp(2)$ such that $\JJ^2=-I$. Let 
$\JJ_0=\begin{pmatrix}
\I & 0\cr
0 &\I \cr
\end{pmatrix}$.
Then the Cartan embedding $\gamma\colon G/K \to Sp(2)$ is $\gamma(\JJ)=-\JJ\JJ_0$. As a consequence, it can be proven that the Cartan model $M$ equals the manifold $N$  of matrices such that $\sigma(A)=A^*$
(Proposition \ref{COMPONENTE}). Explicitly, it is formed by the diagonal matrices
$$\begin{pmatrix}
\alpha & 0 \cr
0 & \delta
\end{pmatrix}, \quad |\alpha| = |\delta |=1, \Re(\alpha\I)=\Re(\delta\I)=0,$$ jointly with the matrices
$$\begin{pmatrix}
\alpha & -\I\bar{\beta}\I\cr
\beta &  \beta\bar{\alpha}\I {\beta^{-1}}\I
\end{pmatrix}, \quad \beta\neq0, |\alpha|^2+|\beta|^2=1, \Re(\alpha\I)=0.$$

Let us take $X=\begin{pmatrix} x & 0 \cr 0 & y \end{pmatrix}\in \HH^{2\times 2}$ with 
  $x=1+\J$ and $y= \I+\J$.
  
First, we study the function $h_X^G$ on the Lie group $G$. Since 
$$
XX^* = 
\begin{pmatrix} |x|^2 & 0 \cr 0 & |y|^2 \end{pmatrix} = 2I,$$ we have the singular value decomposition
$X=UDV^*$ with $U=\frac{1}{\sqrt{2}}X$, $D=\diag(\sqrt{2},\sqrt{2})$ and $V=I$.
The critical set  $\Sigma(h_X^G)$ is then diffeomorphic to $\Sigma(h_D^G)$, which is a disjoint union $\Sigma(2)=G_0^2\sqcup G_1^2 \sqcup G_2^2$ of Grassmannians  \cite{TMP}, namely,  two points and $Sp(2)/(Sp(1)\times Sp(1))\cong S^4$. More precisely, from the gradient equation $D=ADA$  it follows that  the three components are $\{\pm I\}$ and the sphere 
$$\left\{\begin{pmatrix} s & \bar\beta \cr \beta & -s \end{pmatrix}\colon s\in\R, s^2+|\beta|^2=1\right\},$$
which are the orbits by the adjoint action of $I$, $-I$ and $\pm P=\pm\begin{pmatrix}1&0\cr 0&-1\end{pmatrix}$ respectively \cite{FRANKEL}. From the index of each orbit one easily obtains the critical values of the height function.

Finally we have an explicit description of the critical set of $h_X^G$ as $\Sigma(h_X^G)=V\Sigma(h_D^G)U^*$, that is the disjoint union of $\{\pm U^*\}$ and $G_1^2U^*$ (Proposition~\ref{RelCrit}, see also \cite[p.~3]{MPSPN}). It could be computed directly from the gradient equation $X^*=AXA$.
Notice that $\Sigma(h_X^G)\cap M=\emptyset$, that is, none of the critical points of $h_X^G$ is contained in the symmetric space $M\cong G/K$.
 
Now we can compute the critical points of the restriction $h_X^M$ of the height function to $M\subset G$. We shall use the formula  $\Sigma(h_X^M)=\Sigma(h_{\sigma(\HX)}^G)\cap M$. (Corollary~\ref{CRITSIM}), where $\HX=X^*+\sigma(X)$. 

We have 
$$\sigma(\HX)=\HX^*=\begin{pmatrix}  x_0 & 0 \cr 0 & y_0 \end{pmatrix},$$
where $x_0=\bar x-\I x \I=2\J$ and $y_0=\bar y -\I y \I=2+2\J$. Notice that $|x|=|y|$ but $|x_0|\neq |y_0|$.
This time we take the SVD
$\HX^*=UDV^*$ where $U=\begin{pmatrix} \J & 0 \cr 0 & \frac{1}{\sqrt{2}}(1+\J) \end{pmatrix}$, $D=\diag( 2 , 2\sqrt{2})$ and $V=I$. This is an adapted SVD (Corollary~\ref{ADAPT}), that is, the matrix $$\Theta=U^*\sigma(V)=\begin{pmatrix}  -\J & 0 \cr 0 & \frac{1}{\sqrt{2}}(1-\J) \end{pmatrix}$$ verifies $\sigma(\Theta)=\Theta^*$.
 
The critical set $\Sigma(h_D^G)$ is diffeomorphic to $\Sigma(1)\times \Sigma(1)$, that is, four points, because $\Sigma(1)=G_0^1\sqcup G_1^1$ \cite{TMP}. Explicitly,    
$$\Sigma(h_D^G)=\{A\in Sp(2)\colon DA^*=AD\}=\{\pm I, \pm P\}.$$ 
Now, $\Sigma(h_{\sigma(\HX)}^G)=V\Sigma(h_D^G)U^*=\{\pm U^*, \pm PU^*\} $, and these four points are in $M$, then  it follows that $ h_X^M$ is a Morse function with four critical points.
   
Finally, we shall verify the formula 
$\Sigma(h_X^M)=V\Sigma(h_D^{M'})U^*$, where $M^\prime=\Theta M$ (Lemma \ref{DESPLAZ});  remember that $M=N$ in the present example.
 
 According to Lemma \ref{DESPLAZ} the new automorphism  $\sigma '\colon G \to G$ is given by
$$\sigma'(A)=\Theta\sigma(A)\Theta^*=U^*(-\I A\I)U=-\I UAU^*\I$$
and we check that 
$\sigma'(D)=-\I UDU^*\I=-\I D\I=D,$ and 
$\widehat{D}^\prime =D^*+\sigma'(D)=2D$.
 Then 
$$\Sigma(h_D^{M^\prime})=\Sigma(h_{\sigma'(\widehat{D}^\prime)}^G\cap {M^\prime})=\Sigma(h_{2D}^G)\cap M^\prime.$$

We already know that $\Sigma(h_{2D}^{M^\prime})=\Sigma(h_D^{M^\prime})=\{\pm I, \pm P\}$.
On the other hand  $M'=\Theta M=U^*M$, and
$$\Sigma(h_D^G)\cap M^\prime=\{\pm I, \pm P\}.$$
Then 
$$V\Sigma(h_D^{M'})U^*=\{\pm I U^*, \pm PU^*\},$$
which equals $\Sigma(h_X^M)$ as promised.

\bigskip

\flushright
{\sc Enrique Mac\'ias Virg\'os. \\Facultade de Matem\'aticas. \\Universidade de San\-tia\-go de Compostela. SPAIN.}\\
{\tt quique.macias@usc.es}\\

{\sc Mar\'ia Jos\'e Pereira S\'aez. \\Facultade de Econom\'ia e Empresa. \\Universidade da Coru\~na. SPAIN.}\\
{\tt maria.jose.pereira@udc.es}

\end{document}